\documentclass[11pt]{article}
\usepackage{amsfonts}
\usepackage{amsmath,latexsym,amssymb,amsfonts,amsbsy, amsthm}
\usepackage[usenames]{color}
\usepackage{xspace,colortbl}
\usepackage{mathrsfs}
\usepackage[colorlinks,linkcolor=blue,citecolor=blue]{hyperref}
\allowdisplaybreaks

\newcommand{\beq}{\begin{equation}}
\newcommand{\R}{{\mathbb R}}
\newcommand{\eeq}{\end{equation}}
\newcommand{\ben}{\begin{eqnarray}}
\newcommand{\een}{\end{eqnarray}}
\newcommand{\beno}{\begin{eqnarray*}}
\newcommand{\eeno}{\end{eqnarray*}}

\newtheorem{thm}{Theorem}[section]

\newtheorem{lem}[thm]{Lemma}
\newtheorem{prop}[thm]{Proposition}
\newtheorem{coro}[thm]{Corollary}

\title{On Entire Solutions of an Elliptic System Modeling Phase Separations}

\author{Henri Berestycki \thanks{\'{E}cole des hautes \'{e}tudes en sciences sociales, CAMS, 54, boulevard Raspail, F - 75006 - Paris, France. Email: hb@ehess.fr}\ , Susanna Terracini \thanks{Dipartimento di Matematica e Applicazioni, Universita degli Studi di Milano-Bicocca, Piazza Ateneo Nuovo 1, 20126, Milano, Italy. Email: susanna.terracini@unimib.com}, \\
Kelei Wang \thanks{School of Mathematics and Statistics, University of Sydney, Sydney, NSW 2006, Australia. Email:keleiw@mail.usyd.edu.au }, Juncheng Wei \thanks{Department of Mathematics, The Chinese University of Hong Kong, Shatin, Hong Kong. Email: wei@math.cuhk.edu.hk} }

\begin{document}

\maketitle

\date{}

\begin{abstract}
We study the qualitative properties of a limiting elliptic system arising in phase separation for Bose-Einstein condensates with multiple states:
\[
\begin{cases} \Delta u=u v^2\ \ \mbox{in} \ \R^n, \\
\Delta v= v u^2 \ \ \mbox{in} \ \R^n, \\
u, v>0\quad \  \mbox{in} \ \R^n.
\end{cases}
\]
When $n=1$, we prove uniqueness of the one-dimensional profile. In dimension $2$,
we prove that stable solutions  with linear growth must be
one-dimensional. Then we construct entire solutions in $\R^2$
with polynomial growth $|x|^d$ for any positive integer $d \geq 1$. For $d\geq 2$, these solutions are not one-dimensional. The construction is also extended to multi-component elliptic
systems.

\end{abstract}


\vskip 0.2cm


\vskip 0.2cm

\section{Introduction and Main Results}
\setcounter{equation}{0}

Consider the  following two-component Gross-Pitaevskii system
\begin{align}
& -\Delta u + \alpha u^3 + \Lambda v^2 u = \lambda_1 u &&\text{in }\Omega, \label{1}\\
& -\Delta v +\beta v^3 + \Lambda u^2 v = \lambda_2 v  &&\text{in }\Omega, \label{2}\\
& u>0,\quad v>0 && \text{in }\Omega, \label{37}\\
& u=0,\quad v=0 && \text{on }\partial\Omega\,, \label{3}\\
& \int_\Omega u^2=N_1,\quad\int_\Omega v^2=N_2\, , \label{301} &&
\end{align}
where $\alpha, \beta, \Lambda >0$ and $\Omega$ is a bounded smooth domain  in $\R^n$. Solutions of (\ref{1})-(\ref{301})
can be regarded as critical points of the energy  functional
\begin{equation}\label{5.1}
E_\Lambda(u,v)=\int_\Omega\,\left(|\nabla u|^2+|\nabla
v|^2\right)+\frac{\alpha}{2}u^4+\frac{\beta}{2}v^4+\frac{\Lambda}{2}
u^2v^2\,,\end{equation} on the space $(u,v)\in H^1_0(\Omega)\times
H^1_0(\Omega)$ with  constraints
\begin{equation}
\label{302}
\int_\Omega u^2 dx=N_1, \int_\Omega v^2 dx=N_2.
\end{equation}

The eigenvalues $\lambda_j$'s are Lagrange multipliers with
respect to~(\ref{302}). Both eigenvalues
$\lambda_j=\lambda_{j,\Lambda}, j=1,2$, and eigenfunctions $u=u_\Lambda,
v=v_\Lambda$ depend on the parameter $\Lambda$. As the parameter $\Lambda$ tends to infinity, the two components tend to separate their supports. In order to investigate the basic rules of phase separations in this system one needs  to understand the asymptotic behavior of $(u_\Lambda, v_\Lambda)$ as $ \Lambda \to +\infty$.

We shall assume that the solutions $(u_\Lambda,
v_\Lambda)$ of (\ref{1})-(\ref{301}) are such that the associated
eigenvalues $\lambda_{j,\Lambda}$'s are uniformly bounded, together with their energies $ E_\Lambda (u_\Lambda, v_\Lambda)$.
Then, as $\Lambda \to +\infty$, there is weak convergence (up to a subsequence) to a limiting profile $(u_\infty, v_\infty)$ which formally satisfies
\begin{equation}
  \label{eq:limit-equation1}
  \begin{cases}
-\Delta u_{\infty} +\alpha u_{\infty}^3 =\lambda_{1,\infty}
u_{\infty} \qquad & \text{in $\Omega_u$}\,,\\
-\Delta v_{\infty} +\beta v_{\infty}^3 =\lambda_{2,\infty}
v_{\infty} \qquad &\text{in $\Omega_v$}\,,\\
\end{cases}
\end{equation} where $\Omega_u=\{x\in\Omega: u_\infty(x)>0\}$ and
$\Omega_v=\{x\in\Omega: v_\infty(x)>0\}$ are positivity domains
composed of finitely disjoint components with positive Lebesgue
measure, and each $\lambda_{j,\infty}$ is the limit of
$\lambda_{j,\Lambda}$'s as $\Lambda\to\infty$ (up to a
subsequence).

There is a large literature about this type of questions. Effective numerical simulations for
(\ref{eq:limit-equation1}) can be
found in~\cite{B}, \cite{BaD} and~\cite{CLLL}.
Chang-Lin-Lin-Lin ~\cite{CLLL} proved  pointwise convergence of
$(u_\Lambda, v_\Lambda)$ away from the interface
$\Gamma\equiv\{x\in\Omega: u_\infty(x)=v_\infty(x)=0\}$. In  Wei-Weth
\cite{ww}  the uniform equicontinuity of $(u_\Lambda,
v_\Lambda)$ is established, while Noris-Tavares-Terracini-Verzini~\cite{NTTV} proved
the uniform-in-$\Lambda$ H\"older continuity of $(u_\Lambda, v_\Lambda)$. The regularity of the nodal set of the
limiting profile has been investigated in \cite{C-L 2, TT2011} and in \cite{DWZ2011}: it turns out that
the limiting pair  $(u_\infty(x),v_\infty(x))$ is the positive and negative pair $(w^+,w^-)$ of a solution of  the equation $-\Delta w+\alpha (w^{+})^3-\beta (w^{-})^3 =\lambda_{1,\infty}w^+-\lambda_{2,\infty}w^-$.

To derive the asymptotic behavior of $(u_\Lambda, v_\Lambda)$
near the interface $\Gamma=\{x\in\Omega:
u_\infty(x)=v_\infty(x)=0\}$, one is led to considering the points
 $x_\Lambda \in \Omega$ such that $ u_\Lambda
(x_\Lambda)=v_\Lambda (x_\Lambda)= m_\Lambda\to 0$ and $x_\Lambda
\to  x_\infty \in \gamma\subset\Omega$ as $\Lambda \to +\infty$
(up to a subsequence). Assuming that
 \begin{equation}
 \label{mainas}
 m_\Lambda^4 \Lambda \to C_0>0,
 \end{equation}
  (without loss of generality we may assume that $ C_0=1$), then, by blowing up, we find the following  nonlinear
elliptic system
\begin{equation}\label{maineqn}
\Delta u= u v^2\,, \quad \Delta v= v u^2\,, \quad u,
v > 0 \quad \mbox{in} \quad \R^n\,.
\end{equation}

Problem (\ref{maineqn}) has been studied in Berestycki-Lin-Wei-Zhao \cite{blwz}, and Noris-Tavares-Terracini-Verzini  \cite{NTTV}. It has been proved in \cite{blwz} that, in the one-dimensional case, (\ref{mainas}) always holds. In addition, the authors showed the existence, symmetry and nondegeneracy of the
solution to one-dimensional limiting system
\begin{equation}
\label{1D}
u^{''}= uv^2, v^{''}=v u^2, u, v>0 \ \mbox{in} \ \R.
\end{equation}

In particular they showed that entire solutions are reflectionally symmetric, i.e., there exists $x_0$ such that $ u(x-x_0)= v(x_0-x)$.  They  also established  a two-dimensional version of the De Giorgi Conjecture in this framework. Namely, under the growth condition
 \begin{equation}
\label{bd1}
u(x)+v(x)\leq C (1+|x|),
\end{equation}
all monotone solution is one dimensional.

On the other hand, in \cite{NTTV}, it was proved that the linear growth is the lowest possible for solutions to (\ref{maineqn}). In other words, if there exists $\alpha \in (0,1)$ such that
\begin{equation}
\label{bd2}
u(x)+v(x)\leq C (1+|x|)^{\alpha},
\end{equation}
then $u, v \equiv 0$.

In this paper we address three problems left open in \cite{blwz}. First, we
 prove the uniqueness of (\ref{1D}) (up to translations and scaling). This answers
 the question stated  in Remark 1.4 of \cite{blwz}. Second,   we prove that the De Giorgi conjecture still holds in the two dimensional case, when we replace the monotonicity assumption by  the stability condition. A third open question  of (\ref{maineqn}) is whether  all solutions to (\ref{maineqn}) necessarily  satisfy the growth bound (\ref{bd1}). We shall answer this question negatively in this paper.

We first study the one-dimensional  problem (\ref{1D}).
 Observe that problem (\ref{1D}) is invariant under the translations $ (u(x), v(x)) \to ( u(x+t), v(x+t)), \forall t \in \R$ and scalings $ (u(x), v(x)) \to ( \lambda u(\lambda x), \lambda  v(\lambda x)), \forall \lambda >0$. The following theorem  classifies all entire solutions to (\ref{1D}).

\begin{thm}
\label{thm0}
The solution to (\ref{1D}) is unique, up to translations and scaling.
\end{thm}

 Next,we want to classify the stable solutions in $\R^2$. We recall that a {\em stable} solution $(u, v)$ to (\ref{maineqn}) is such that the linearization is weakly positive definite. That is, it satisfies
\[
\int_{\R^n} [\nabla \varphi|^2+|\nabla \psi |^2 + v^2 \varphi^2+u^2 \psi^2 +4 uv \varphi \psi] \geq 0, \qquad \forall \varphi, \psi \in C_0^\infty (\R^n).
\]

In \cite{blwz}, it was proved that the one-dimensional solution is stable in $\R^n$. Our first result states that the only stable solution in $\R^2$, among those growing at most linearly,  is the one-dimensional family.

\begin{thm}
\label{thm1}
Let $(u,v)$ be a stable solution to (\ref{maineqn}) in $\R^2$. Furthermore, we assume that the growth bound (\ref{bd1}) holds. Then $(u, v)$ is one-dimensional, i.e., there exists $a \in \R^2, |a|=1$ such that $(u, v)= (U (a \cdot x), V (a \cdot x))$ where $(U, V)$ are functions of one variable and satisfies (\ref{1D}).

\end{thm}

Our third result shows that there are solutions to (\ref{maineqn})
with polynomial growth $|x|^d$ that are not one dimensional.  The construction depends on  the following harmonic polynomial $\Phi$ of degree
$d$:
$$\Phi:=\mbox{Re}(z^d).$$
Note that $\Phi$ has some dihedral symmetry; indeed, let us take its
$d$ nodal lines $L_1, \cdots, L_d$ and denote the corresponding
reflection with respect to these lines by $T_1,\cdots, T_d$. Then there holds
\begin{equation}\label{reflectional symmetry}
\Phi(T_i z)=-\Phi(z).
\end{equation}
The third result of this paper is the following one.
\begin{thm}\label{main result}
For each positive integer $d \geq 1$, there exists a solution $(u,v)$ to  problem \eqref{maineqn}, satisfying
\begin{enumerate}
\item $u-v>0$ in $\{\Phi>0\}$ and $u-v<0$ in $\{\Phi<0\}$;
\item $u \geq\Phi^+$ and $v\geq\Phi^-$;
\item $\forall i=1,\cdots, d$, $u(T_iz)=v(z)$;
\item $\forall r>0$, the Almgren frequency function satisfies
\begin{equation}
\label{nr} N(r):=\frac{r\int_{B_r(0)}|\nabla u|^2+|\nabla
v|^2+u^2v^2}{\int_{\partial B_r(0)}u^2+v^2}\leq d;
\end{equation}
\item
\begin{equation}\label{nr 2}
 \lim_{r \to +\infty} N(r) =d.
\end{equation}
\end{enumerate}
\end{thm}

Note that the one-dimensional solution constructed in \cite{blwz} can be viewed as corresponding to the case $d=1$. For $d\geq 2$, the solutions of Theorem \ref{main result} will be obtained by a minimization argument under symmetric variations
$(\varphi,\psi)$ (i.e. satisfying $\varphi\circ T_i=\psi$ for every
reflection $T_i$). The first four claims will be derived from  the
construction. See Theorem \ref{thm existence on bounded set}.

\medskip

Regarding the claim 5, we note that by  Almgren's monotonicity formula,
(see Proposition \ref{monotonocity} below), the Almgren frequency quotient $N(r)$
is increasing in $r$. Hence $ \lim_{r \to +\infty} N(r)$ exists.
To understand the asymptotics at infinity of the solutions,
one way is to study the
blow-down sequence defined by:
$$(u_R(x), v_R(x)):=(\frac{1}{L(R)}u(Rx)\frac{1}{L(R)}v(Rx)),$$
where $L(R)$ is chosen so that
$$\int_{\partial
B_1(0)}u_R^2+v_R^2=1.$$
 In Section 6, we will prove
\begin{thm}\label{thm asymptotics at infinity}
Let $(u,v)$ be a solution of \eqref{maineqn} such that
\[d:=\lim\limits_{r\rightarrow+\infty}N(r)<+\infty.\]
Then $d$ is a positive integer.
As $R\to\infty$, $(u_R, v_R)$ defined above (up to a subsequence)
converges to $(\Psi^+,\Psi^-)$ uniformly on any compact set of
$\mathbb{R}^N$ where $\Psi$ is a homogeneous harmonic polynomial of
degree $d$. If $d=1$ then $(u,v)$ is asymptotically flat at infinity.
\end{thm}
In particular
this applies to the solutions found by Theorem \ref{main result} to yield the following property
\begin{coro}
Let $(u,v)$ be a solution of \eqref{maineqn} given by Theorem \ref{main result}. Then
$$(u_R(x), v_R(x)):=(\frac{1}{R^d}u(Rx)\frac{1}{R^d}v(Rx))$$
converges uniformly on compact subsets of $\mathbb R^2$  to a multiple of $(\Phi^+,\Phi^-)$, where
$\Phi:=\mbox{Re}(z^d)$.
\end{coro}
\par
Theorem \ref{thm asymptotics at infinity} roughly says that $(u,v)$
is asymptotic to $(\Psi^+,\Psi^-)$ at infinity for some homogeneous harmonic polynomial. The extra information we have in the setting of Theorem \ref{main result} is that $\Psi\equiv\Phi=\mbox{Re}(z^d)$. This can be inferred  from the symmetries of the solution (property $3$ in Theorem \ref{main result}).

For another elliptic system with a similar form,
\begin{equation}
\label{uvnew}
\left\{ \begin{aligned}
 &\Delta u=uv, u>0 \ \mbox{in} \ \R^n,\\
 &\Delta v=vu, v>0 \ \mbox{in} \ \R^n                         \end{aligned} \right.
\end{equation}
the same result has been proved by Conti-Terracini-Verzini in \cite{C-T-V 3}.
In fact, their result hold for any dimension $n\geq 1$ and any
harmonic polynomial function on $\mathbb{R}^n$. Note however that the problem here is different from (\ref{uvnew}). Actually, equation (\ref{uvnew})  can be reduced to a single equation: indeed, the difference $u-v$ is a harmonic function ($\Delta (u-v)=0$) and thus we can write $v= u-\Phi $ where $\Phi$ is a harmonic function. By restricting to certain symmetry classes, then (\ref{uvnew}) can be solved by sub-super solution method. However, this reduction does not work for system (\ref{maineqn}) that we study here.

For the proof of Theorem \ref{main result}, we first construct solutions to (\ref{maineqn})  in any bounded ball $B_R(0)$ satisfying appropriate boundary conditions:
\begin{equation}\label{equation100}
\left\{ \begin{aligned}
 &\Delta u=uv^2, ~~\mbox{in}~~B_R(0),\\
 &\Delta v=vu^2,~~\mbox{in}~~B_R(0), \\
& u=\Phi^+,  v=\Phi^- \ \mbox{ on} \ \partial B_R(0).
                          \end{aligned} \right.
\end{equation}

This is done by variational  method  and using  heat flow. The next natural step is to let
$R\rightarrow+\infty$ and obtain some convergence result. This requires
some uniform (in $R$) upper bound for solutions to
(\ref{equation100}). In order to prove
this, we will exploit a new monotonicity formula for symmetric  functions (Proposition \ref{prop:upperbound}).
We also need to exclude the possibility of
degeneracy,  that is  that the limit could be $0$ or a solution with lower
degree such as a one dimensional solution. To this end, we will give some lower
bound using the Almgren monotonicity formula.

Lastly, we observe that  the same construction works also for  a  system with many
components. Let $d$ be an integer or a half-integer and $2d=hk$ be a
multiple of the number of components $k$, and $G$ denote the
rotation of order $2d$. In this way we prove the following result

\begin{thm}\label{thm:maini}
There exists a positive solution to the system
\begin{equation}\label{eq:system}
\left\{ \begin{aligned}
 &\Delta u_i=u_i\sum_{j\neq i,j=1}^ku_j^2, ~~\mbox{in}~~\mathbb C=\R^2, i=1,\dots, k,\\
  & u_i>0, i=1,\ldots, k,
                          \end{aligned} \right.
  \end{equation}
having the following symmetries (here $\overline{z}$ is the complex conjugate of $z$)
\begin{equation}
\label{eqn2_i}
\begin{aligned}
u_{i}(z)&=u_i(G^hz), \qquad \ &\mbox{ on} \ &\mathbb C\,,i=1,\dots,k,\\
u_i(z)&=u_{i+1}(Gz), \qquad \ &\mbox{ on} \ &\mathbb C\,,i=1,\dots,k,\\
u_{k+1}(z)&=u_1(z),  \ &\mbox{ on} \ &\mathbb C\\
u_{k+2-i}(z)&=u_i(\overline{z}), \qquad \ &\mbox{ on} \ &\mathbb C\,,i=1,\dots,k.\\
\end{aligned}
\end{equation}
Furthermore,
\[\lim_{r\to\infty} \dfrac{1}{r^{1+2d}}\int_{\partial B_r(0)}\sum_1^k u_{i}^2=b\in(0,+\infty)\;;\]
and
\[\lim_{r\to\infty} \frac{r\int_{B_r(0)}\sum_1^k |\nabla u_{i}|^2+\sum_{i<j}u_i^2u_j^2}
{\int_{\partial B_r(0)}\sum_1^k u_{i}^2}=d\;.\]
\end{thm}

The problem of the full classification of solutions to
\eqref{maineqn} is largely open. In view of our results, one can formulate several open questions.

 \medskip

 \noindent
 {\bf Open problem 1.} We recall from \cite{blwz} that it is still an open problem to know in which dimension it is true that  all monotone solution is one-dimensional. A similar  open question is  in which dimension it is true that  all stable solution is one-dimensional.  We refer to \cite{A-C}, \cite{GG}, \cite{dkw}, \cite{pacard}, and \cite{savin} for results of this kind for Allen-Cahn equation.

 \medskip

 \noindent
 {\bf Open problem 2.} Let us recall that in one space
dimension, there exists a unique solution to (\ref{1D}) (up to translations and scalings). Such solutions have linear growth at infinity and, in the Almgren
monotonicity formula, they satisfy
\begin{equation}
\label{2.1n}
\lim\limits_{r\rightarrow+\infty}N(r)=1.
\end{equation}
It is natural to conjecture that, in any space dimension, a solution of (\ref{maineqn}) satisfying (\ref{2.1n}) is actually  one dimensional, that is, there is a unit vector
$a$  such that $(u(x),v(x))=(U (a \cdot x ), V (a \cdot x))$ for
$x \in\mathbb{R}^n$, where $(U,V)$ solves (\ref{1D}). However this result seems to be difficult to obtain at this stage.

 \medskip

 \noindent
 {\bf Open problem 3.} A further step would be to  prove uniqueness  of the (family of) solutions having polynomial asymptotics given by Theorem \ref{main result} in two space dimension. A more challenging question is to classify all solutions with
 \begin{equation}
\label{2.1nn}
\lim\limits_{r\rightarrow+\infty}N(r)=d.
\end{equation}

\medskip

\noindent
{\bf Open problem 4.}  For the Allen-Cahn equation $ \Delta u+u-u^3=0$ in $\R^2$, solutions similar to Theorem \ref{main result} was first constructed in \cite{dfp} for $d=2$ and in \cite{acm}  for $ d\geq 3$. (However all solutions to Allen-Cahn equation are bounded.)  On the other hand, it was also proved in \cite{dkpw} that Allen-Can equation in $\R^2$ admits solutions with multiple fronts. An open question is whether similar result holds for (\ref{maineqn}). Namely, are there solutions to (\ref{maineqn}) such that the set $\{ u=v\}$ contains disjoint multiple curves?

\medskip

\noindent
{\bf Open problem 5.} This question is related to extension of Theorem \ref{main result} to higher dimensions.  We recall that for the Allen-Cahn equation $ \Delta u+u-u^3=0$ in $\R^{2m}$ with $m\geq 2$, saddle-like solutions were constructed in \cite{cabre} by employing properties of Simons cone. Stable solutions to Allen-Cahn equation in $\R^8$ with non planar level set were found in \cite{pacard}, using minimal cones. We conjecture that all these results should have analogues for (\ref{maineqn}).

\section{Uniqueness of solutions in $\R$: Proof of Theorem \ref{thm0}}
\numberwithin{equation}{section}
 \setcounter{equation}{0}

In this section we prove Theorem \ref{thm0}.   Without loss of generality, we assume that
\begin{equation}
 \lim_{x \to +\infty} u(x)= +\infty, \lim_{x \to +\infty} v(x)=0.
\end{equation}
The existence of such entire solutions has been proved in \cite{blwz}.  By symmetry property of solutions to  (\ref{1D}) (Theorem 1.3 of \cite{blwz}), we may consider the following problem
\begin{equation}\label{entire problem}
\left\{ \begin{aligned}
 &u^{''}=uv^2,v^{''}=vu^2, u,v>0~~\text{in}~~\mathbb{R},\\
 &\lim\limits_{x\to+\infty}u^{'}(x)=-\lim\limits_{x\to-\infty}v^{'}(x)=a
                          \end{aligned} \right.
\end{equation}
where $a>0$ is a constant. We now prove that there exists a unique
solution $(u, v)$ to (\ref{entire problem}), up to translations. We
will prove it using the method of moving planes.

First we observe that for any solution $(u,v)$ of \eqref{entire problem}, $u^{''}$ and
$v^{''}$ decay exponentially at infinity. Integration shows that as
$x\to+\infty$, $|u^{'}(x)-a|$ decays exponentially. (See also \cite{blwz}.) This implies the
existence of a positive constant $A$ such that
\begin{equation}\label{uniform deviation}
|u(x)-ax^+|+|v(x)-ax^-|\leq A.
\end{equation}
Moreover, the limits
\[\lim\limits_{x\to+\infty}(u(x)-ax^+),\lim\limits_{x\to-\infty}(v(x)-ax^-)\]
exist.

Now assume $(u_1,v_1)$ and $(u_2,v_2)$ are two solutions of
\eqref{entire problem}. For $t>0$, denote
\[u_{1,t}(x):=u_1(x+t),v_{1,t}(x):=v_1(x+t).\]
We want to prove that there exists an optimal  $t_0$ such that for all $t\geq
t_0$,
\begin{equation}\label{sliding}
u_{1,t}(x)\geq u_2(x),v_{1,t}(x)\leq
v_2(x)~~\text{in}~~\mathbb{R}.
\end{equation}
Then we will show that  when $t=t_0$ these inequalities are identities. This will
imply the uniqueness result.

Without loss of generality, assume $(u_1,v_1)$ and $(u_2,v_2)$
satisfy the estimate \eqref{uniform deviation} with the same
constant $A$.

\medskip

\noindent
{\bf Step 1.} For $t\geq \frac{16A}{a}$ ($A$ as in \eqref{uniform deviation}), \eqref{sliding} holds.

\medskip

Firstly, in the region $\{x\geq -t+\frac{2A}{a}\}$, by \eqref{uniform deviation} we have
\begin{equation}\label{1n}
u_{1,t}(x)\geq a(x+t)-A\geq ax^++A\geq u_2(x);
\end{equation}
while in the region $\{x\leq -t+\frac{2A}{a}\}$, we have
\begin{equation}\label{2n}
v_{1,t}(x)\leq a(x+t)^-+A\leq ax^--A\leq v_2(x).
\end{equation}
On the interval  $\{x<-t+\frac{2A}{a}\}$, we have
\begin{equation}
\label{new1}
\left\{ \begin{aligned}
 &u_{1,t}^{''}=u_{1,t}v_{1,t}^2\leq u_{1,t}v_2^2,\\
 &u_2^{''}=u_2v_2^2.\\
                          \end{aligned} \right.
\end{equation}
With the right boundary conditions
\[u_{1,t}(-t+\frac{2A}{a})\geq u_2(-t+\frac{2A}{a}),
 \lim\limits_{x\to-\infty}u_{1,t}(x)=\lim\limits_{x\to-\infty}u_2(x)=0,\]
 a direct application of the
maximum principle implies
\[\inf_{\{x<-t+\frac{2A}{a}\}}(u_{1,t}-u_2)\geq 0.\]
By the same type of argument also show that
\[\sup_{\{x>-t+\frac{2A}{a} \}}(v_{1,t}-v_2)\leq 0.\]

Therefore, we have shown that for $ t\geq \frac{16A}{a}, u_{1,t}\geq u_2$ and $ v_{1,t}\leq v_2 $.

\medskip

\noindent
{\bf Step 2.} We now decrease the $t$ to an optimal value when (\ref{sliding}) holds
\[t_0=\inf\{t^{'} | \ \text{such that}~~\eqref{sliding}~~\text{holds} ~~\text{for all} ~~ t \geq t^{'} \}.\]
Thus $t_0$ is well defined by Step 1.  Since $ -(u_{1,t_0}-u_2)^{''} +v_{1,t_0}^2  (u_{1,t_0}-u_2) \geq 0,\ -(v_{2}-v_{1,t_0})^{''} +u_{1,t_0}^2  (v_2- v_{1,t_0}) \geq 0,$ by the strong maximum principle, either
\[u_{1,t_0}(x)\equiv u_2(x),v_{1,t_0}(x)\equiv
v_2(x)~~\text{in}~~\mathbb{R},\] or
\begin{equation}
\label{new23}
u_{1,t_0}(x)> u_2(x),v_{1,t_0}(x)<
v_2(x)~~\text{in}~~\mathbb{R}.
\end{equation}

Let us argue by contradiction that (\ref{new23}) holds. By the definition of $t_0$, there
exists a sequence of $t_k<t_0$ such that
$\lim\limits_{k\to+\infty}t_k=t_0$ and  either
\begin{equation}
\label{2.0n}
\inf_{\mathbb{R}}(u_{1,t_k}-u_2)<0,
\end{equation}
or
\[\sup_{\mathbb{R}}(v_{1,t_k}-v_2)>0.\]
Let us only consider  the first case.

Define $w_{1,k}:=u_{1,t_k}-u_2$ and $w_{2,k}:=v_2-v_{1,t_k}$. Direct
calculations show that they satisfy
\begin{equation}\label{2.9}
\left\{ \begin{aligned}
 &-w_{1,k}^{''}+v_{1,t_k}^2w_{1,k}=u_2(v_2+v_{1,t_k})w_{2,k}~~\mbox{in}~~\mathbb{R}, \\
 &-w_{2,k}^{''}+u_{1,t_k}^2w_{2,k}=v_2(u_2+u_{1,t_k})w_{1,k}~~\mbox{in}~~\mathbb{R}.
                          \end{aligned} \right.
\end{equation}

We use the auxiliary function $g(x)=\log(|x|+3)$ as in \cite{cl}. Note that
$$g\geq 1,~~g^{''}<0~~\mbox{in}~~\{x\neq 0\}.$$
Define $\widetilde{w}_{1,k}:=w_{1,k}/g$ and $\widetilde{w}_{2,k}:=w_{2,k}/g$. For $x \not = 0$ we have
\begin{equation}\label{2.10}
\left\{ \begin{aligned}
 &-\widetilde{w}_{1,k}^{''}-2\frac{g^{'}}{g}\widetilde{w}_{1,k}^{'}
 +[v_{1,t_k}^2-\frac{g^{'}}{g}]\widetilde{w}_{1,k}=u_2(v_2+v_{1,t_k})\widetilde{w}_{2,k},
 ~~\mbox{in}~~\mathbb{R}, \\
 &-\widetilde{w}_{2,k}^{''}-2\frac{g^{'}}{g}\widetilde{w}_{2,k}^{'}
 +[u_{1,t_k}^2-\frac{g^{'}}{g}]\widetilde{w}_{2,k}=v_2(u_2+u_{1,t_k})\widetilde{w}_{1,k},~~\mbox{in}~~\mathbb{R}.
                          \end{aligned} \right.
\end{equation}
By definition,  $w_{1,k}$ and $w_{2,k}$ are bounded in $\mathbb{R}$, and hence
\[\widetilde{w}_{1,k},\widetilde{w}_{2,k}\to 0~~\text{as}~~|x|\to\infty.\]
In particular, in view of (\ref{2.0n}), we know that $\inf_{\R} (\widetilde{w}_{1,k})<0$ is attained at some point
$x_{k,1}$.

Note that $|x_{k,1}|$ must be unbounded, for if $ x_{k,1} \to x_{\infty}, t_k\to t_0$, then $ w_{1,k} (x_{k,1}) \to  u_{1,t_0} (x_\infty)- u_2 (x_\infty) =0$. But this violates the assumption (\ref{new23}).

Since $|x_{k,1}|$ is unbounded,  at $x=x_{k,1}$  there holds
$$\widetilde{w}_{1,k}^{''}\geq 0~~\mbox{and}~~\widetilde{w}_{1,k}^{'}=0.$$
Substituting this into the first equation of \eqref{2.10}, we get
\begin{equation}\label{2.11}
[v_{1,t_k}(x_{k,1})^2-\frac{ g^{''}(x_{k,1})}{g(x_{k,1})}]
\widetilde{w}_{1,k}(x_{k,1})\geq
u_2(x_{k,1})(v_2(x_{k,1})+v_{1,t_k}(x_{k,1}))\widetilde{w}_{2,k}(x_{k,1})
\end{equation}
which implies that $\widetilde{w}_{2,k}(x_{k,1})<0$.  Thus we also have
$\inf\limits_{\mathbb{R}}\widetilde{w}_{2,k}<0$. Assume it is attained
at $x_{k,2}$. Same argument as before shows that $|x_{k,2}|$ must also be unbounded.  Similar to \eqref{2.11}, we have
\begin{equation}\label{2.12}
[u_{1,t_k}(x_{k,2})^2-\frac{g^{''}(x_{k,2})}{g(x_{k,2})}]
\widetilde{w}_{2,k} (x_{k,2})\geq
v_2(x_{k,2})(u_2(x_{k,2})+u_{1,t_k}(x_{k,2}))\widetilde{w}_{1,k} (x_{k,2}).
\end{equation}

Observe  that
$$\widetilde{w}_{2,k} (x_{k,2})=\inf\limits_{\mathbb{R}}\widetilde{w}_{2,k} \leq\widetilde{w}_{2, k} (x_{k,1}),
$$
$$\widetilde{w}_{1, k} (x_{k,1})=\inf\limits_{\mathbb{R} }\widetilde{w}_{1,k} \leq\widetilde{w}_{1, k} (x_{k,2}).
$$
Substituting these into \eqref{2.11} and \eqref{2.12}, we obtain
\begin{equation}
\label{2.13n}
\widetilde{w}_{1, k} (x_{k,1})\geq
\frac{u_2(x_{k,1})[v_2(x_{k,1})+v_{1,t_k}(x_{k,1})]}
{v_{1,t_k}(x_{k,1})^2-\frac{g^{''}(x_{k,1})}{g(x_{k,1})}}
\frac{v_2(x_{k,2})[u_2(x_{k,2})+u_{1,t_k}(x_{k,2})]}
{u_{1,t_k}(x_{k,2})^2-\frac{g^{''}(x_{k,2})}{g(x_{k,2})}}
\widetilde{w}_{1,k}(x_{k,1}).
\end{equation}

Since $ \tilde{w}_{1,k} (x_{k,1}) <0$, we conclude from (\ref{2.13n}) that
\begin{equation}
\frac{u_2(x_{k,1})[v_2(x_{k,1})+v_{1,t_k}(x_{k,1})]}
{v_{1,t_k}(x_{k,1})^2-\frac{g^{''}(x_{k,1})}{g(x_{k,1})}}
\frac{v_2(x_{k,2})[u_2(x_{k,2})+u_{1,t_k}(x_{k,2})]}
{u_{1,t_k}(x_{k,2})^2-\frac{g^{''}(x_{k,2})}{g(x_{k,2})}} \geq 1
\end{equation}
where $|x_{k,1}|\to +\infty, |x_{k,2}| \to +\infty$. This is impossible since $ \frac{ g^{''} (x)}{g (x)} \sim -\frac{1}{|x|^2 \log (|x|+3)}$ as $|x| \to +\infty$, and we also use the decaying as well as the linear growth properties of $u$ and $v$ at $\infty$.

We have thus reached  a contradiction, and the proof of Theorem \ref{thm0} is thereby completed.

\section{Stable solutions: Proof of Theorem \ref{thm1}}
\numberwithin{equation}{section}
 \setcounter{equation}{0}

In this section, we prove Theorem \ref{thm1}. The proof follows  an idea from  Berestycki-Caffarelli-Nirenberg \cite{BCN}-see also Ambrosio-Cabr\'{e} \cite{A-C} and Ghoussoub-Gui \cite{GG}.  First, by the stability, we have the following
\begin{lem}
There exist a constant $\lambda\geq 0$ and two functions $\varphi>0$
and $\psi<0$, smoothly defined in $\mathbb{R}^2$ such that
\begin{equation}\label{entire eigenfunction}
\left\{ \begin{aligned}
 &\Delta\varphi=v^2\varphi+2uv\psi-\lambda\varphi,\\
 &\Delta\psi= 2uv\varphi+v^2\psi-\lambda\psi.
                          \end{aligned} \right.
\end{equation}
\end{lem}
\begin{proof}
For any $R<+\infty$ the stability assumption reads
$$\lambda(R):=\min\limits_{\varphi,\psi\in H_0^1(B_R(0))\setminus\{0\}}
\frac{\int_{B_R(0)}|\nabla\varphi|^2+|\nabla\psi|^2+v^2\varphi^2+u^2\psi^2+4uv\varphi\psi}
{\int_{B_R(0)}\varphi^2+\psi^2}\geq0.$$ It's well known that the
corresponding minimizer is the first eigenfunction. That is, let
$(\varphi_R,\psi_R)$ realizing $\lambda(R)$, then
\begin{equation}\label{first eigenfunction R}
\left\{ \begin{aligned}
 &\Delta\varphi_R= v^2\varphi_R+2uv\psi_R-\lambda(R)\varphi_R, \ \mbox{in} \ B_R (0),\\
 &\Delta\psi_R= 2uv\varphi_R+v^2\psi_R-\lambda(R)\psi_R, \ \mbox{in} \ B_R (0), \\
 & \varphi_R =\psi_R=0 \ \mbox{on} \ \partial B_R (0).
                          \end{aligned} \right.
\end{equation}
By possibly replacing $(\varphi_R,\psi_R)$ with $(|\varphi_R|,-|\psi_R|)$, we
can assume $\varphi_R\geq 0$ and $\psi_R\leq 0$. After a
normalization, we also assume
\begin{equation}\label{normalization condition 2}
|\varphi_R(0)|+|\psi_R(0)|=1.
\end{equation}
 $\lambda(R)$ is
decreasing in $R$, thus uniformly bounded as $R\rightarrow+\infty$.
Let $$\lambda:=\lim\limits_{R\rightarrow+\infty}\lambda(R).$$ The
equation for $\varphi_R$ and $-\psi_R$ (both of them are
nonnegative functions) forms a cooperative system, thus by the Harnack
inequality (\cite{A-G-M} or \cite{BS}), $\varphi_R$ and $\psi_R$ are uniformly
bounded on any compact set of $\mathbb{R}^2$. By letting
$R\rightarrow+\infty$, we can obtain a converging subsequence and the
limit $(\varphi,\psi)$ satisfies \eqref{entire eigenfunction}.
\par
We also have $\varphi\geq 0$ and $\psi\leq 0$ by passing to the
limit. Hence
$$-\Delta\varphi+(v^2-\lambda)\varphi\geq 0.$$
Applying the strong maximum principle, either $\varphi>0$ strictly
or $\varphi\equiv 0$. If $\varphi\equiv 0$, substituting this into
the first equation in \eqref{entire eigenfunction}, we obtain
$\psi\equiv 0$. This contradicts the normalization condition
\eqref{normalization condition 2}. Thus, it holds true that $\varphi>0$ and
similarly $\psi<0$.
\end{proof}
Fix a unit vector $\xi$. Differentiating the equation (\ref{maineqn}) yields the following equation for $(u_{\xi},v_{\xi})$
\begin{equation}\label{linearized equation}
\left\{ \begin{aligned}
 &\Delta u_{\xi}=v^2u_{\xi}+2uvv_{\xi},\\
 &\Delta v_{\xi}=2uvu_{\xi}+v^2v_{\xi}.
                          \end{aligned} \right.
\end{equation}
Let
$$w_1=\frac{u_{\xi}}{\varphi},w_2=\frac{v_{\xi}}{\psi}.$$
Direct calculations using \eqref{entire eigenfunction} and
\eqref{linearized equation} show
\begin{equation}
\left\{ \begin{aligned}
 &\text{div}(\varphi^2\nabla
 w_1)=2uv\varphi\psi(w_2-w_1)+\lambda\varphi^2w_1,\\\nonumber
 &\text{div}(\varphi^2\nabla w_2)=2uv\varphi\psi(w_1-w_2)+\lambda\psi^2w_2.
                          \end{aligned} \right.
\end{equation}
For any $\eta\in C_0^{\infty}(\mathbb{R}^2)$, testing these two
equations with $w_1\eta^2$ and $w_2\eta^2$ respectively, we obtain
\begin{equation}
\left\{ \begin{aligned}
 &-\int\varphi^2|\nabla w_1|^2\eta^2-2\varphi^2w_1\eta\nabla w_1\nabla\eta
 =\int
 2uv\varphi\psi(w_2-w_1)w_1\eta^2+\lambda\varphi^2w_1\eta^2,\\\nonumber
 &-\int\psi^2|\nabla w_2|^2\eta^2-2\psi^2w_2\eta\nabla w_2\nabla\eta
 =\int 2uv\varphi\psi(w_1-w_2)w_2\eta^2+\lambda\psi^2w_2\eta^2.
                          \end{aligned} \right.
\end{equation}
Adding these two and applying the Cauchy-Schwarz inequality, we infer that
\begin{equation}\label{5.6}
\int\varphi^2|\nabla w_1|^2\eta^2+\psi^2|\nabla w_2|^2\eta^2\leq 16
\int\varphi^2w_1^2|\nabla\eta|^2+\psi^2w_2^2|\nabla\eta|^2\leq 16
\int (u_{\xi}^2+v_{\xi}^2)|\nabla\eta|^2.
\end{equation}
Here we have taken away the positive term in the right hand side and used
the fact that
$$2uv\varphi\psi(w_2-w_1)w_1\eta^2+2uv\varphi\psi(w_1-w_2)w_2\eta^2
=-2uv\varphi\psi(w_1-w_2)^2\eta^2\geq0,$$ because $\varphi>0$ and
$\psi<0$.

\medskip

On the other hand, testing  the equation $\Delta u \geq 0$ with $u\eta^2$ ($\eta$ as
above) and integrating by parts, we get
$$\int|\nabla u|^2\eta^2\leq 16\int u^2|\nabla\eta|^2.$$
The same estimate also holds for $v$. For any $r>0$, take
$\eta\equiv 1$ in $B_r(0)$, $\eta\equiv0$ outside $B_{2r}(0)$ and
 $|\nabla\eta|\leq 2/r$. By the linear growth of $u$ and $v$, we obtain
 a constant $C$ such that
\begin{equation}\label{5.7}
\int_{B_r(0)}|\nabla u|^2+|\nabla v|^2\leq Cr^2.
\end{equation}
Now for any $R>0$, in \eqref{5.6}, we take $\eta$ to be
\begin{equation}
\eta(z)= \left\{
\begin{array}{ll}
1, &  x\in B_R(0), \\\nonumber
 0, & x\in B_{R^2}(0)^c,\\\nonumber
 1-\frac{\log(|z|/R)}{\log R}  & x\in B_{R^2}(0)\setminus B_R(0).
\end{array}
\right. \end{equation} With this $\eta$, we infer from (\ref{5.6})
\begin{eqnarray*}
&&\int_{B_R(0)}\varphi^2|\nabla w_1|^2+\psi^2|\nabla w_2|^2\\
&\leq&\frac{C}{(\log R)^2}\int_{B_{R^2}(0)\setminus
B_R(0)}\frac{1}{|z|^2}(|\nabla u|^2+|\nabla v|^2)\\
&\leq &\frac{C}{(\log R)^2}\int_R^{R^2}r^{-2}(\int_{\partial
B_r(0)}|\nabla u|^2+|\nabla
v|^2)dr\\
&=&\frac{C}{(\log
R)^2}\int_R^{R^2}r^{-2}(\frac{d}{dr}\int_{B_r(0)}|\nabla
u|^2+|\nabla
v|^2)dr\\
&=&\frac{C}{(\log R)^2}[r^{-2}\int_{\partial B_r(0)}|\nabla
u|^2+|\nabla v|^2)|_R^{R^2}+2\int_R^{R^2}r^{-3}(\int_{B_r(0)}|\nabla
u|^2+|\nabla v|^2)dr]\\
&\leq& \frac{C}{\log R}.
\end{eqnarray*}
By letting $R\rightarrow+\infty$, we see $\nabla w_1\equiv 0$ and
$\nabla w_2\equiv 0$ in $\mathbb{R}^2$. Thus, there is a constant $c$
such that
$$(u_{\xi},v_{\xi})=c(\varphi,\psi).$$
Because $\xi$ is an arbitrary unit vector, from this we actually know that after
changing the coordinates suitably,
$$u_y\equiv 0,v_y\equiv 0\ \ \text{in}~\mathbb{R}^2.$$
That is, $u$ and $v$ depend on $x$ only and they are one
dimensional.

\section{Existence in bounded balls}
\numberwithin{equation}{section}
 \setcounter{equation}{0}
In this section we first construct a solution $(u,v)$ to the problem
\begin{equation}\label{equation}
\left\{ \begin{aligned}
 &\Delta u=uv^2 ~~\mbox{in}~~B_R(0),\\
 &\Delta v=vu^2 ~~\mbox{in}~~B_R(0),
                          \end{aligned} \right.
\end{equation}
satisfying the boundary condition
\begin{equation}
\label{eqn2}
u=\Phi^+,  v=\Phi^- \ \mbox{ on} \ \partial B_R(0)\subset\mathbb{R}^2.
\end{equation}

More precisely, we prove
\begin{thm}\label{thm existence on bounded set}
There exists a solution $(u_R,v_R)$ to  problem \eqref{equation},
satisfying
\begin{enumerate}
\item $u_R-v_R>0$ in $\{\Phi>0\}$ and $u_R-v_R<0$ in $\{\Phi<0\}$;
\item $u_R\geq\Phi^+$ and $v_R\geq\Phi^-$;
\item $\forall i=1,\cdots, d$, $u_R(T_iz)=v_R(z)$;
\item $\forall r\in(0,R)$,
$$N(r;u_R,v_R):=\frac{r\int_{B_r(0)}|\nabla u_R|^2+|\nabla v_R|^2+u_R^2v_R^2}
{\int_{\partial B_r(0)}u_R^2+v_R^2}\leq d.$$
\end{enumerate}
\end{thm}

\begin{proof}

Let us denote $\mathcal U\subset H^1(B_R(0))^2$ the set of pairs
satisfying the boundary condition \eqref{eqn2}, together with
conditions $(1,2,3)$ of the statement of the Theorem (with the
strict inequality $<$ replaced by $\leq$, and so now $\mathcal{U}$ is a
closed set).

 The desired solution will be a minimizer of the energy  functional
$$E_R(u,v):=\int_{B_R(0)}|\nabla u|^2+|\nabla v|^2+u^2v^2$$
over $\mathcal U$. Existence of at least one minimizer follows easily from the direct method of the Calculus of Variations.   To prove that the minimizer also satisfies equation (\ref{equation}), we use the heat flow method.
More precisely, we consider the following parabolic problem
\begin{equation}\label{parabolic equation}
\left\{ \begin{aligned}
 &U_t-\Delta U=-UV^2, ~~\mbox{in}~~[0,+\infty)\times B_R(0),\\
 &V_t-\Delta V=-VU^2,~~\mbox{in}~~[0,+\infty)\times B_R(0),
                          \end{aligned} \right.
\end{equation}
with the boundary conditions $U=\Phi^+$ and $V=\Phi^{-}$ on
$(0,+\infty)\times \partial B_R(0)$ and initial conditions in
$\mathcal U$.

\medskip

By the standard  parabolic theory, there exists a unique local solution
$(U,V)$. Then by the maximum principle, $0\leq U\leq
\sup_{B_R(0)}\Phi^+, \ \ 0\leq V \leq \sup_{B_R(0)} \Phi^{-}$,
 hence the solution can be extended to a
global one, for all $t\in(0,+\infty)$. By noting the energy inequality
\begin{equation}
\label{ert}
\frac{d}{dt}E_R(U(t),V(t))=-\int_{B_R(0)}|\frac{\partial U}{\partial t}|^2
+|\frac{\partial V}{\partial t}|^2
\end{equation}
and the fact that $E_R\geq 0$, standard parabolic theory implies that for any sequence $t_i\to+\infty$, there
exists a subsequence of $t_i$ such that $(U(t_i),V(t_i))$ converges to
a solution $(u,v)$ of \eqref{equation}.

\medskip

Next we show that $\mathcal U$ is positively invariant by the parabolic flow.
First of all, by the symmetry of initial and boundary data, $(V(t,T_iz),U(t,T_iz))$ is also a solution to the problem
\eqref{parabolic equation}. By the uniqueness of solutions to the
parabolic system \eqref{parabolic equation}, $(U,V)$ inherits the
symmetry of $(\Phi^+,\Phi^-)$. That is, for all $t\in[0,+\infty)$
and $i=1,\cdots, d$,
$$U(t,z)=V(t,T_iz).$$  
This implies
$$U-V=0~~\mbox{on}~~\{\Phi=0\}.$$
Thus, in the open set $D_R:=B_R(0)\cap\{\Phi>0\}$, we have, for any initial datum $(u_0,v_0)\in\mathcal U$,
\begin{equation}\label{eq:difference}
\left\{ \begin{aligned}
 &(U-V)_t-\Delta (U-V)=UV(U-V), ~~\mbox{in}~~[0,+\infty)\times D_R(0),\\
 &U-V\geq 0,~~\mbox{on}~~[0,+\infty)\times \partial D_R(0),\\
 &U-V\geq 0,~~\mbox{on}~~\{0\}\times D_R(0).
                          \end{aligned} \right.
\end{equation}
The strong maximum principle implies $U-V>0$ in $(0,+\infty)\times
D_R(0)$. By letting $t\to+\infty$, we obtain that the limit satisfies
\begin{equation}\label{4.10}
u-v\geq 0~~\mbox{in}~~D_R(0).
\end{equation}
$(u,v)$ also has the symmetry, $\forall i=1,\cdots, d$
$$u (T_i z)=v (z).$$

\medskip

 Similar to \eqref{eq:difference}, noting \eqref{4.10}, we have
\begin{equation}
\left\{ \begin{aligned}
 &-\Delta (u-v)\geq 0, ~~\mbox{in}~~D_R(0),\\
 &u-v=\Phi^+,~~\mbox{on}~~\partial D_R(0).
                          \end{aligned} \right.
\end{equation}
Comparing with $\Phi^+$ on $D_R(0)$, we obtain
\begin{equation}\label{2m}
u-v>\Phi^+>0,~~\mbox{in}~~D_R(0).
\end{equation}

 Because $u>0$ and $v>0$ in $B_R(0)$, we in
fact have
\begin{equation}\label{3n}
u>\Phi^+,~~\mbox{in}~~B_R(0).
\end{equation}
In conclusion, $(u,v)$ satisfies conditions $(1,2,3)$ in the
statement of the theorem.

Let $(u_R, v_R)$ be a minimizer of $E_{R}$ over ${\mathcal U}$.  Now we consider the parabolic equation (\ref{parabolic equation}) with the initial condition
\begin{equation}
U(x, t)= u_R (x), V (x, t) = v_R (x).
\end{equation}
By (\ref{ert}), we deduce that
$$ E_R (u_R, v_R) \leq  E_{R} (U, V) \leq E_R (u_R, v_R) $$
and hence $ (U(x, t), V(x, t) )\equiv (u_R(x), v_R (x))$ for all $ t \geq 0$. By the arguments above, we see that $ (u_R, v_R)$ satisfies (\ref{equation})and  conditions $(1,2,3)$ in the statement of the theorem.

In order to prove (4), we firstly note that, as $(u_R,v_R)$ minimizes the energy and $(\Phi^+,\Phi^-)\in\mathcal U$, there holds
\[\int_{B_R(0)}|\nabla u_R|^2+|\nabla v_R|^2+u_R^2v_R^2\leq
\int_{B_R(0)}|\nabla \Phi|^2.
\]
Now by the Almgren
monotonicity formula (Proposition \ref{monotonocity} below) and the boundary conditions, $\forall
r\in(0,R)$, we derive
$$N(r;u_R,v_R)\leq N(R;u_R,v_R)\leq \frac{R\int_{B_R(0)}|\nabla\Phi|^2}{\int_{\partial B_R(0)}|\Phi|^2}=d.$$
This completes the proof of Theorem \ref{thm existence on bounded set}.

\end{proof}


Let us now turn to the system with many components.
In a similar way we shall prove the existence on bounded sets.  Let $d$ be an integer or a
half-integer and $2d=hk$ be a multiple of the number of components
$k$, and $G$ denote the rotation of order $2d$. Take the fundamental
domain $F$ of the rotations group of degree $2d$, that is
$F=\{z\in\mathbb{C}\;:\;
\theta=\mbox{arg}(z)\in(-\pi/{2d},\pi/{2d})\}$.

\begin{equation}
\Psi(z)=\begin{cases}
r^{d}\cos(d\theta)\qquad&\text{if $z\in \cup _{i=0}^{h-1}G^{ik}(F)$,}\\
0 & \text{otherwise in $\mathbb C$.}
\end{cases}
\end{equation}
Note that $\Psi(z)$ is positive whenever it is not zero. Next we construct a solution $(u_1,\dots,u_k)$ to the system
\begin{equation}\label{equation_i}
 \Delta u_i=u_i\sum_{j\neq i,j=1}^ku_j^2, ~~\mbox{in}~~B_R(0), i=1,\dots, k
\end{equation}
satisfying the symmetry and boundary condition (here $\overline{z}$ is the complex conjugate of $z$)
\begin{equation}
\label{eqn2_in}
\left\{\begin{array}{l}
u_{i}(z)= u_i(G^hz), \qquad \ \ \ \ \mbox{ on} \ B_R(0)\,,i=1,\dots,k,\\
u_i (z)= u_{i+1}(Gz), \qquad \ \ \mbox{ on} \ B_R(0)\,,i=1,\dots,k,\\
u_{k+2-i}(z)= u_i(\overline{z}), \qquad \ \mbox{ on} \ B_R(0)\,,i=1,\dots,k,\\
u_{k+1}(z)= u_1(z),  \ \ \ \ \ \ \ \ \ \mbox{ on} \ B_R(0),
\end{array}
\right.
\end{equation}
\begin{equation}
\label{eqn2_ibc}
u_{i+1}(z)=\Psi(G^i(z)),  \qquad \mbox{ on} \ \partial B_R(0)\,,i=0,\dots,k-1.
\end{equation}

More precisely, we prove the following.
\begin{thm}\label{thm existence on bounded seti}
For every $R>0$, there exists a solution $(u_{1,R},\dots,u_{k,R})$ to the system \eqref{equation_i} with symmetries \eqref{eqn2_in} and boundary conditions \eqref{eqn2_ibc},
satisfying,
$$N(r):=\frac{r\int_{B_r(0)}\sum_1^k|\nabla u_{i,R}|^2+\sum_{i<j}u_{i,R}^2u_{j,R}^2}
{\int_{\partial B_r(0)}\sum_1^k u_{i,R}^2}\leq d, \ \forall r\in(0,R).$$
\end{thm}
\begin{proof}
Let us denote by $\mathcal U\subset H^1(B_R(0))^k$ the set of pairs satisfying the symmetry and boundary condition \eqref{eqn2_in}, \eqref{eqn2_ibc}. The desired solution will be the minimizer of the energy  functional
$$\int_{B_r(0)}\sum_1^k|\nabla u_{i,R}|^2+\sum_{i<j}u_{i,R}^2u_{j,R}^2$$
over $\mathcal U$. Once more, to deal with the  constraints, we may take advantage of the positive invariance of the associated heat flow:
\begin{equation}\label{parabolic equation_i}
\left\{ \begin{aligned}
 &\dfrac{\partial U_i}{\partial t}-\Delta U_i=-U_i\sum_{j\neq i}U_j^2, ~~\mbox{in}~~[0,+\infty)\times B_R(0),\\
                          \end{aligned} \right.\end{equation}
which can be solved under conditions \eqref{eqn2_i}, \eqref{eqn2_ibc} and initial conditions in $\mathcal U$.
Thus, the minimizer of the energy $(u_{1,R},\dots,u_{k,R})$ solves the differential system.
In addition, using the test function $(\Psi_1,\dots,\Psi_k)$, where $\Psi_i=\Psi\circ G^{i-1}$, $i=1,\dots,k$, we have
\[\int_{B_R(0)}\sum_1^k|\nabla u_{i,R}|^2+\sum_{i<j}u_{i,R}^2u_{j,R}^2\leq
k\int_{B_R(0)}|\nabla \Psi|^2.
\]
Now by the Almgren
monotonicity formula below (Proposition \ref{monotonocity}) and the boundary conditions, we get
$$N(r)\leq N(R)\leq \frac{R\int_{B_R(0)}|\nabla\Psi|^2}{\int_{\partial B_R(0)}|\Psi|^2}=d, \ \forall r \in (0, R).$$
\end{proof}

In order to conclude the proof of Theorems \ref{thm existence on bounded set} and \ref{thm existence on bounded seti}, we need to find upper and lower bounds for the solutions, uniform with respect to $R$ on bounded subsets of $\mathbb C$. That
is, we will prove that for any $ r>0$, there  exists positive constants $0<c(r)<C(r)$ (independent of $R$) such
that
\begin{equation}\label{uniform upper bound}c(r)<
\sup\limits_{B_r(0)}u_R\leq C(r).
\end{equation}
Once  we have this estimate, then by letting $R\rightarrow+\infty$, a
subsequence of $(u_R,v_R)$ will converge to a solution $(u,v)$ of
 problem (\ref{maineqn}), uniformly on any compact set of
$\mathbb{R}^2$. It is easily seen that properties (1), (2), (3) and
(4) in Theorem \ref{thm existence on bounded set} can be  derived by passing to
the limit, and we obtain the main results stated in Theorem \ref{main result} and \ref{thm:maini}. It then remains to establish the bound (\ref{uniform upper bound}). In the next section, we shall obtain this estimate by using the monotonicity formula.


\section{Monotonicity formula}
Let us start by stating  some monotonicity formulae for solutions to
(\ref{maineqn}), for any dimension $n\geq 2$.  The first two are
well-known and we include them here for completeness. But we will also require some refinements.

\begin{prop}\label{monotonocity 1}
For $r>0$ and $x\in\mathbb{R}^n$,
$$E(r)=r^{2-n}\int_{B_r(x)}\sum_1^k|\nabla u_i|^2+\sum_{i<j}u_i^2u_j^2$$
is nondecreasing in $r$.
\end{prop}
For a proof, see \cite{C-L 2}. The next statement is an Almgren-type
monotonicity formula with remainder.
\begin{prop}\label{monotonocity}
For $r>0$ and $x\in\mathbb{R}^n$, let us define
\[H(r)=r^{1-n}\int_{\partial B_r(x)}\sum_1^k u_i^2.\] Then
$$N(r;x):=\frac{E(r)}{H(r)}$$
is nondecreasing in $r$. In addition there holds
\begin{equation}\label{eq:remainder}
\int_{0}^r \dfrac{2\int_{B_s}\sum_{i<j}^ku_i^2u_j^2}{\int_{\partial B_s}\sum_1^ku_i^2}ds\leq N(r)\;.
\end{equation}
\end{prop}
\begin{proof}
For simplicity, take $x$ to be the origin $0$ and let $k=2$. We have
\[H(r)= r^{1-n}\int_{\partial B_r}u^2+v^2\;,\qquad E(r)=r^{2-n}\int_{B_r}|\nabla u|^2+|\nabla v|^2+u^2v^2\;.
\]
Then, direct calculations show that
\begin{equation}\label{4.2}
\frac{d}{dr} H(r)
=2r^{1-n}\int_{B_r}|\nabla u|^2+|\nabla v|^2+2u^2v^2.
\end{equation}
By the proof of
Proposition \ref{monotonocity 1}, we have
\begin{equation}\label{4.1}
\frac{d}{dr} E(r)
=2r^{2-n}\int_{\partial B_r}[u_r^2+v_r^2] +2r^{1-n}\int_{B_r}u^2v^2.
\end{equation}
With these two identities, we obtain
\begin{multline*}
\frac{d}{dr}\frac{E}{H}(r)
=\dfrac{H [2r^{2-n}\int_{\partial B_r}(u_r^2+v_r^2)
+2r^{1-n}\int_{B_r}u^2v^2]
-E[2r^{1-n}\int_{\partial B_r} uu_r+vv_r]}{H^2}\\
\geq \dfrac{2r^{3-2n}\int_{\partial B_r}(u^2+v^2)
\int_{\partial B_r}(u_r^2+v_r^2)
-2r^{3-2n}\left[\int_{\partial B_r} uu_r+vr_r\right]^2}{H^{2}}+\\
+\dfrac{2r^{1-n}\int_{B_r}u^2v^2}{H}
\geq \dfrac{2r^{1-n}\int_{B_r}u^2v^2}{H}.
\end{multline*}
Here we have used  the following inequality
$$E(r)\leq \int_{B_r}|\nabla u|^2+|\nabla v|^2+2u^2v^2
=\int_{\partial B_r} uu_r+vr_r.$$
Hence this yields monotonicity of the Almgren quotient. In addition, by integrating the above
inequality we obtain
\begin{equation*}
\int_{r_0}^r \dfrac{2\int_{B_s}u^2v^2}{\int_{\partial B_s}u^2+v^2}ds\leq N(r)\;.
\end{equation*}
\end{proof}
If $x=0$, we simply denote $N(r;x)$ as $N(r)$. Assuming an upper
bound on $N(r)$, we establish a doubling property by the Almgren
monotonicity formula.
\begin{prop}\label{doubling property}
Let $R>1$ and let $(u_1,\dots,u_k)$ be a solution of \eqref{eq:system} on $B_R$. If $N(R)\leq d$, then for any $1<r_1\leq r_2\leq R$
\begin{equation}\label{eq:h_monotone}
\dfrac{H(r_2)}{H(r_1)}\leq e^{d}\dfrac{r_2^{2d}}{r_1^{2d}}.
\end{equation}
\end{prop}
\begin{proof}
For simplicity of notation, we expose the proof for the case of two components. By direct calculation using \eqref{4.2}, we obtain
\begin{eqnarray*}
\frac{d}{dr}\log\Bigg[r^{1-n} (\int_{\partial B_r(0)}u^2+v^2)\Bigg]
&=&\frac{2\int_{B_r}|\nabla u|^2+|\nabla v|^2+2u^2v^2}{\int_{\partial B_r(0)}u^2+v^2}\\
&\leq& \frac{2N(r)}{r}+\frac{2\int_{B_r}u^2v^2}{\int_{\partial B_r(0)}u^2+v^2}\\
&\leq& \frac{2d}{r}+\frac{2\int_{B_r}u^2v^2}{\int_{\partial B_r(0)}u^2+v^2}\\
\end{eqnarray*}
Thanks to \eqref{eq:remainder}, by integrating, we find that, if $r_1\leq r_2\leq 2r_0$ then
\begin{equation}\label{eq:h_monotone1}
\dfrac{H(r_2)}{H(r_1)}\leq e^{d}\dfrac{r_2^{2d}}{r_1^{2d}}.
\end{equation}
\end{proof}
An immediate consequence of Proposition \ref{doubling property} is the lower bound on bounded sets for the solutions found in Theorems \ref{thm existence on bounded set} and \ref{thm existence on bounded seti}.
\begin{prop}\label{prop:lowerbound}
Ler $(u_{1,R},\dots,u_{k,R})$ be a family of solutions to \eqref{eq:system} such that $N(R)\leq d$ and $H(R)= CR^{2d}$. Then, for every fixed $r<R$, there holds
\[H(r)\geq Ce^{-d}r^{2d}.\]
\end{prop}
Another byproduct of the monotonicity formula with the remainder \eqref{eq:remainder} is the existence of the limit of $H(r)/r^{2d}$.
\begin{coro}\label{existencelimitH}
Let $R>1$ and let $(u_1,\dots,u_k)$ be a solution of \eqref{eq:system} on $\mathbb C$ such that $\lim_{r\to+\infty}N(r)\leq d$, then there exists
\begin{equation}\label{eq:Hhaslimit}
\lim_{r\to+\infty}\dfrac{H(r)}{r^{2d}}<+\infty\;.
\end{equation}
\end{coro}

%

Now we prove the optimal lower bound on the growth of the solution. To this aim, we need a fine estimate on the asymptotics of the
lowest eigenvalue as the competition term diverges. The following result is an extension of Theorem 1.6 in \cite{blwz}, where
the estimate was proved in case of two components.
\begin{thm}\label{thm:lambda}
Let $d$ be a fixed integer and let us consider
\begin{multline}\label{eq:min_eigenvalue}
\mathcal{L}(d,\Lambda)=\min\left\{ \int_0^{2\pi}\sum_i^d|u^\prime_i|^2+\Lambda\sum_{i<j}^d u_i^2u_j^2\; \Bigg| \; \begin{array}{l} \int_0^{2\pi}\sum_i u_i^2=1, \
 u_{i+1}(x)=u_i(x-2\pi/d),\\
  u_1(-x)=u_1(x)\;,u_{d+1}=u_1\;
 \end{array}
 \right\}.
\end{multline}
Then, there exists a constant $C$ such that for all $\Lambda>1$ we have
\begin{equation}\label{eq:est_l}
d^2-C \Lambda^{-1/4}\leq\mathcal{L}(d,\Lambda)\leq d^2\;.
\end{equation}
\end{thm}
\begin{proof}
Any minimizer $(u_{1,\Lambda},\dots,u_ {d,\Lambda})$  solves the system of ordinary differential equations
\begin{equation}u_i^{''}=\Lambda u_i\sum_{j\neq i}u_j^2-\lambda u_i\;,\qquad i=1,\dots,d,
\end{equation}
together with the associated energy conservation law
\begin{equation}
\sum_1^d(u_i^{'})^2+\lambda u_i^2-\Lambda\sum_{i<j}^du_i^2u_j^2=h\;.
\end{equation}
Note that the Lagrange multiplier satisfies
\[\lambda=\int_0^{2\pi}\sum_i^d|u^\prime_i|^2+2\Lambda\sum_{i<j}^d u_i^2u_j^2=\mathcal{L}(d,\Lambda)+\int_0^{2\pi}\Lambda\sum_{i<j}^d u_i^2u_j^2\;.\]
As $\Lambda\to\infty$, we see convergence of the eigenvalues $\lambda\simeq \mathcal{L}(d,\Lambda)\to d^2$, together with the energies $h\to 2d^2$. Moreover, the solutions remain bounded in Lipschitz norm and converge in Sobolev and H\"older spaces (see \cite{blwz} for more details).
Now, let us focus on the interval $I=(a,a+2\pi/d)$ where the $i$-th component is active.
The symmetry constraints imply
\begin{multline*}u_{i-1}(a)=u_i(a)\;,u_{i-1}^{'}(a)=-u_i^{'}(a)\;,\\
 u_{i+1}(a+2\pi/d)=u_i(a+2\pi/d)\;,u_{i+1}^{'}(a+2\pi/d)=-u_i^{'}(a+2\pi/d)\end{multline*}
We observe that there is interaction only with the two prime neighboring components, while the others are exponentially small  (in $\Lambda$) on $I$.
Close to the endpoint $a$, the component  $u_i$ is increasing and convex, while $u_{i-1}$ is decreasing and again convex. Similarly to \cite{blwz} we have that
\begin{equation}\label{eq:iv}
u_i(a)= u_{i-1}(a)\simeq K\Lambda^{-1/4}\;, u'_i(a)= -u^{'}_{i-1}(a)\simeq H=(h+K)/2\;.
\end{equation}
Hence, in a right neighborhood of $a$, there holds
$u_i(x)\geq u_i(a)$, and therefore, as $u_{i-1}^{''}\geq\Lambda u_i^2(a)u_{i-1}$, from the initial value problem \eqref{eq:iv} we infer
\[u_{i-1}(x)\leq C u_i(a)e^{-\Lambda^{1/2}u_i(a)(x-a)}\;,\forall x\in [a,b].\]
On the other hand, on the same interval we have
\[u_{i}(x)\leq  u_i(a)+C(x-a)\;,\forall x\in [a,b].\]
(here and below $C$ denotes a constant independent of $\Lambda$). Consequently, there holds
\begin{equation}\label{eq:allterms}
\Lambda \int_I u_{i-1}^2 u_{i}^2+u_{i-1}^3 u_{i}+u_{i-1} u_{i}^2\leq C\Lambda^{-1/2}u_i(a)^{-1}\simeq C\Lambda^{-1/4}\;.
\end{equation}
In particular, this yields
\begin{equation}\label{eq:lambda}
\mathcal L(d,\Lambda)\geq \lambda -C\Lambda^{-1/4}\;.
\end{equation}
In order to estimate $\lambda$, let us consider $\widehat u_i=\left(u_i-\sum_{j=i\pm 1}u_j\right)^+$.
Then, as $u_i(a)=u_{i-1}(a)$ and $u_i(a+2\pi/d)=u_{i+1}(a+2\pi/d)$, $\widehat u_i\in H^1_0(I)$. By testing the differential equation for $u_i-\sum_{j=i\pm 1}u_j$ with $\widehat u_i$ on $I$ we find
\[\int_I |\widehat u_i^{'}|^2\leq \lambda \int_I |\widehat u_i|^2+C\Lambda^{-1/4}\;,\]
where in the last term we have majorized all the integrals of mixed fourth order monomials with \eqref{eq:allterms}.
As $|I|=2\pi/d$, using Poincar\'e inequality and \eqref{eq:lambda} we obtain the desired estimate on $\mathcal L(d,\Lambda)$.
\end{proof}

We are now ready to apply the estimate from below on $\mathcal L$ to derive a lower bound on the energy growth. We recall that  there holds
\[
\widehat E(r):=\int_{B_r(x)}\sum_1^k|\nabla u_i|^2+2\sum_{i<j}u_i^2u_j^2=\int_{\partial B_r(x)}\sum_1^k u_i\dfrac{\partial u_i}{\partial r}
\]
\begin{prop}\label{prop:upperbound}
Let $(u_{1,R},\dots,u_{k,R})$ be a solution of \eqref{eq:system} having the symmetries \eqref{eqn2_i} on $B_R$. There exists a constant $C$ (independent of $R$) such that for all $\;1\leq r_1\leq r_2\leq R$ there holds
\begin{equation}
\dfrac{\widehat E(r_2)}{\widehat E(r_1)}\geq C\dfrac{r_2^{2d}}{r_1^{2d}}
\end{equation}
\end{prop}

\begin{proof}
Let us compute,
\begin{multline*}\dfrac{d}{dr}\log\left(r^{-2d}\widehat E(r)\right)=-\dfrac{2d}{r}+\dfrac{\int_{\partial B_r(x)}{\sum_1^k|\nabla u_i|^2}+2\sum_{i<j}u_i^2u_j^2}{\int_{\partial B_r(x)}\sum_1^k u_i\dfrac{\partial u_i}{\partial r}}\\
=-\dfrac{2d}{r}+\dfrac{\int_{\partial B_r(x)}\sum_1^k\left(\dfrac{\partial u_i}{\partial r}\right)^2+\dfrac{1}{r^2}\left[\sum_1^k\left(\dfrac{\partial u_i}{\partial \theta}\right)^2+2r^2\sum_{i<j}u_i^2u_j^2\right]}{\int_{\partial B_r(x)}\sum_1^k u_i\dfrac{\partial u_i}{\partial r}}\\
=-\dfrac{2d}{r}+\dfrac{\int_{0}^{2\pi}\sum_1^k\left(\dfrac{\partial u_i}{\partial r}\right)^2+\dfrac{1}{r^2}\left[\sum_1^k\left(\dfrac{\partial u_i}{\partial \theta}\right)^2+2r^2\sum_{i<j}u_i^2u_j^2\right]}{\int_{0}^{2\pi}\sum_1^k u_i\dfrac{\partial u_i}{\partial r}}\end{multline*}
Now we use Theorem \ref{thm:lambda} and we continue the chain of inequalities:
\begin{multline}\label{eq:Emonotone}\dfrac{d}{dr}\log\left(r^{-2d}\widehat E(r)\right)\geq -\dfrac{2d}{r}+\dfrac{\int_{0}^{2\pi}\sum_1^k\left(\dfrac{\partial u_i}{\partial r}\right)^2+\dfrac{\mathcal L(d,2r^2)}{r^2}\int_0^{2\pi}\sum_1^k u_i^2}{\int_{0}^{2\pi}\sum_1^k u_i\dfrac{\partial u_i}{\partial r}}\\
\geq  -\dfrac{2d-2 \sqrt{\mathcal L(d,2r^2)}}{r}\geq -\dfrac{C}{r^{3/2}}\;,
\end{multline}
where in the last line we have used H\"older inequality. By integration we easily obtain the assertion.
\end{proof}
A direct consequence of the above inequalities is the non vanishing of the quotient $E/r^{2d}$:
\begin{coro}\label{existencelimitE}
Let $R>1$ and let $(u_1,\dots,u_k)$ be a solution of \eqref{eq:system} on $\mathbb C$ satisfying \ref{eqn2_i}: then there exists
\begin{equation}\label{eq:hatEhaslimit}
\lim_{r\to+\infty}\dfrac{\widehat E(r)}{r^{2d}}=b\in (0,+\infty]\;.
\end{equation}
If, in addition, $\lim_{r\to+\infty}N(r)\leq d$, then we have that $b<+\infty$ and
\begin{equation}\label{eq:Ehaslimit}
\lim_{r\to+\infty}N(r)=d,\quad\text{and}\quad\lim_{r\to+\infty}\dfrac{E(r)}{r^{2d}}=b\;.
\end{equation}
\end{coro}
\begin{proof}
Note that \eqref{eq:hatEhaslimit} is a straightforward consequence of the monotonicity formula \eqref{eq:Emonotone}. To prove \eqref{eq:Ehaslimit}, we first notice that
\[\lim_{r\to+\infty}\dfrac{E(r)}{r^{2d}}=\lim_{r\to+\infty}N(r)\dfrac{H(r)}{r^{2d}}.\]
So the limit of $E(r)/r^{2d}$ exists finite. Now we use \eqref{eq:remainder}
\[
\int_{0}^{+\infty} \dfrac{2\int_{B_s}\sum_{i<j}^ku_i^2u_j^2}{\int_{\partial B_s}\sum_1^ku_i^2}ds<+\infty
\]
and we infer
\[\liminf_{r\to+\infty}\dfrac{r\int_{B_{r}}\sum_{i<j}^ku_i^2u_j^2}{\int_{\partial B_{r}}\sum_1^ku_i^2}=0.\]
Next, using Corollary \ref{existencelimitH} we can compute
\[\liminf_{r\to+\infty}\dfrac{\int_{B_{r}}\sum_{i<j}^ku_i^2u_j^2}{r^{2d}}=
\liminf_{r\to+\infty}\dfrac{\int_{B_{r}}\sum_{i<j}^ku_i^2u_j^2}{H(r)} \dfrac{H(r)}{r^{2d}}
=0,\]
and finally
\[\liminf_{r\to +\infty}\dfrac{\widehat E(r)-E(r)}{r^{2d}}=0;.\]
Was the limit of $N(r)$ strictly less that $d$, the growth of $H(r)$ would be in contradiction with that of $E(r)$.
\end{proof}
Now we can combine  the upper and lower estimates to obtain convergence
of the approximating solutions on compact sets and complete the
proof of Theorems \ref{thm:maini}
\begin{proof}[Proof of Theorem \ref{thm:maini}.]
Let $(u_{1,R},\dots,u_{k,R})$ be a family of solutions to \eqref{eq:system} such that $N_R(R)\leq d$ and $H_R(R)= CR^{2d}$.
Since $H_R(R)=CR^{2d}$, then,
by Proposition \ref{doubling property}
 we deduce that, for every fixed $1<r<R$, there holds
\[H_R(r)\geq Ce^{-d}r^{2d}\;. \]

Assume first that there holds a uniform bound for some $r>1$,
\begin{equation}\label{eq:boundonH}
H_R(r)\leq C\;.
\end{equation}

 Then $H_R(r)$ and $E_R(r)$ are uniformly bounded on $R$. This implies a uniform bound on
  the $H^1(B_{r})$ norm. As the components are subharmonic, standard elliptic estimates (Harnack inequality) yield actually a $\mathcal C^2$ bound on $B_{r/2}$, which is independent on $R$. Note that, by Proposition \ref{prop:lowerbound}, $H_R(r)$ is bounded away from zero, so the weak limit cannot be zero. By the doubling Property \ref{doubling property} the uniform bound on $H_R(r_2)\leq C r_2^{2d}$ holds for every $r_2\in\mathbb R$ larger than $r$. Thus, a  diagonal procedure yields existence of a nontrivial limit solution of the differential system, defined on the whole of $\mathbb C$. It is worthwhile noticing that this solution inherits all the symmetries of the approximating solutions together with the upper bound on the Almgren's quotient.
Finally, from Corollary \ref{existencelimitH} and \ref{existencelimitE}  infer the limit
\begin{equation}\label{right growth rate}
\lim_{r\to +\infty}\dfrac{H(r)}{r^{2d}}=\lim_{r\to
+\infty}\dfrac{1}{N(r)},
\lim_{r\to
+\infty}\dfrac{E(r)}{r^{2d}}=\dfrac bd\in(0,+\infty)\:.
\end{equation}

Let us now show that $H_R (r)$ is uniformly bounded with respect to $R$ for fixed $r$. We argue by contradiction and assume that, for a sequence $R_n\to+\infty$, there holds
\begin{equation}\label{eq:Hunbounded}
\lim_{n\to+\infty}H_{R_n}(r)=+\infty\;.
\end{equation}
Denote $u_{i,n}=u_{i,R_n}$ and $H_n$,  $E_n$, $N_n$ the corresponding functions. Note that, as $E_n$ is bounded, we must have $N_n(r)\to 0$. For each $n$, let $\lambda_n\in(0,r)$ such that
\[\lambda^2_nH_n(\lambda_n)=1\;\]
(such $\lambda_n$ exist right because of \eqref{eq:Hunbounded}) and scale
\[\tilde u_{i,n}(z)=\lambda_n u_{i,n}(\lambda_n z)\;, \quad |z|<R_n/\lambda_n\;.\]
Note that the $(\tilde u_{i,n})_i$ still solve system \eqref{eq:system} on the disk $B(0,R_n/\lambda_n)$ and enjoy all the symmetries \eqref{eqn2_i}. Let us denote $\tilde  H_n$, $\tilde E_n$, $\tilde N_n$ the corresponding quantities. We have
\[\begin{aligned}
\tilde  H_n(1)&=\lambda^2_nH_n(\lambda_n)=1,  \\
\tilde E_n(1)&=\lambda^2_nE_n(\lambda_n) \to 0 \\
\tilde N_n(1)&=N_n(\lambda_n)\to 0
\end{aligned}\]
In addition there holds $\tilde N_n(s)\leq d$ for $s<R_n/\lambda_n$. By the compactness argument exposed above, we can extract a subsequence converging in the compact-open topology of $\mathcal C^2$ to a nontrivial symmetric solution of \eqref{eq:system} with Almgren quotient  vanishing constantly. Thus, such solution should be a nonzero constant in each component, but  constant solution are not compatible with the system of PDE's  \eqref{eq:system} .
\end{proof}

\section{Asymptotics at infinity}
\numberwithin{equation}{section}
 \setcounter{equation}{0}
 We  now come to the proof of Theorem \ref{thm asymptotics at
infinity}. Note that by Proposition \ref{doubling property}, the
condition on $N(r)$ implies that $u$ and $v$ have a polynomial
growth. (In fact, with more effort we can show the reverse also holds. Namely, if $u$ and $v$ have polynomial growth, then $N(r)$ approaches a positive integer as $r\to +\infty$. We leave out the proof.)

\medskip

 Recall the blow down
sequence is defined by
$$(u_R(x), v_R(x)):=(\frac{1}{L(R)}u(Rx),\frac{1}{L(R)}v(Rx)),$$
where $L(R)$ is chosen so that
\begin{equation}\label{normalization condition}
\int_{\partial B_1(0)}u_R^2+v_R^2=\int_{\partial B_1(0)}\Phi^2.
\end{equation}
For the solutions in Theorem \ref{main result}, by \eqref{right
growth rate}, we have
\begin{equation}\label{eq:L(R)}
L(R)\sim R^d.
\end{equation}
\par
 We
will now analyze  the limit of $(u_R,v_R)$ as $R\rightarrow+\infty$.
\par
Because for any $r\in(0,+\infty)$, $N(r)\leq d$, $(u, v)$ satisfies
Proposition \ref{doubling property} for any $r\in(1,+\infty)$. After
rescaling, we see that Proposition \ref{doubling property} holds for
$(u_R,v_R)$ as well. Hence, there exists a constant $C>0$, such that for
any $R$ and $r\in(1,+\infty)$,
\begin{equation}\label{4.3}
\int_{\partial B_r(0)}u_R^2+v_R^2\leq  C e^{d}r^d.
\end{equation}
Next, $(u_R,v_R)$ satisfies the equation
\begin{equation}\label{4.4}
\left\{ \begin{aligned}
 &\Delta u_R=L(R)^2R^2u_Rv_R^2,\\
 &\Delta v_R=L(R)^2R^2v_Ru_R^2,\\
 &u_R,v_R>0~~\mbox{in}~~\mathbb{R}^2.
                          \end{aligned} \right.
\end{equation}
Here we need to observe that, by \eqref{eq:L(R)},
$$\lim\limits_{R\rightarrow+\infty}L(R)^2R^2=+\infty.$$

By \eqref{4.3}, as $R\rightarrow+\infty$, $u_R$ and $v_R$ are
uniformly bounded on any compact set of $\mathbb{R}^2$. Then by the
main result in \cite{DWZ2011}, \cite{NTTV} and \cite{TT2011}, there is a harmonic
function $\Psi$ defined in $\mathbb{R}^2$, such that (a subsequence
of) $(u_R,v_R)\rightarrow(\Psi^+,\Psi^-)$ in $H^1$ and in H\"older spaces on any compact
set of $\mathbb{R}^2$. By \eqref{normalization condition},
$$\int_{\partial B_1(0)}\Psi^2=\int_{\partial B_1(0)}\Phi^2,$$
so $\Psi$ is nonzero. Because $L(R)\rightarrow+\infty$, $u_R(0)$ and
$v_R(0)$ goes to $0$, hence
\begin{equation}\label{4.6}
\Psi(0)=0.
\end{equation}
\par
After rescaling in Proposition \ref{monotonocity}, we obtain a
corresponding monotonicity formula for $(u_R,v_R)$,
$$N(r;u_R,v_R):=\frac{r\int_{B_r(0)}|\nabla u_R|^2+|\nabla v_R|^2+L(R)^2R^2u_R^2v_R^2}
{\int_{\partial B_r(0)}u_R^2+v_R^2}=N(Rr)$$ is nondecreasing in $r$.
By (4) in Theorem \ref{main result} and from Corollary \ref{existencelimitE},
\begin{equation}\label{4.7}
N(r;u_R,v_R)\leq d=\lim_{r\to+\infty} N(r;u_R,v_R)\;\;, \forall\; r\in(0,+\infty).
\end{equation}
In \cite{DWZ2011}, it's also proved that
$(u_R,v_R)\rightarrow(\Psi^+,\Psi^-)$ in $H^1_{loc}$ and for any
$r<+\infty$,
$$\lim\limits_{R\rightarrow+\infty}\int_{B_r(0)}L(R)^2R^2u_R^2v_R^2=0.$$
After letting $R\rightarrow+\infty$ in \eqref{4.7}, we get
\begin{equation}\label{convergence of degree}
N(r;\Psi):=\frac{r\int_{B_r(0)}|\nabla \Psi|^2} {\int_{\partial
B_r(0)}\Psi^2}=\lim\limits_{R\rightarrow+\infty}N(r;u_R,v_R)=\lim\limits_{R\rightarrow+\infty}N(Rr)=
d.
\end{equation}
 In
particular, $N(r;\Psi)$ is a constant for all $r\in(0,+\infty)$. So $\Psi$ is a homogeneous polynomial of degree $d$. Actually the number
 $d$ is the vanishing order of $\Psi$ at $0$, which must therefore be a positive integer.
Now it remains  to prove that $\Psi\equiv\Phi$: this is easily done by exploiting the symmetry conditions on $\Psi$ (point $(3)$ of Theorem \ref{main result}).

\bigskip

\noindent {\bf Acknowledgment.} Part of this work was carried out while Henri Berestycki was visiting the
Department of Mathematics at the University of Chicago. Heá was supported
by an NSF FRG grant DMS-1065979 and by the French "Agence Nationale de la
Recherche" within the project PREFERED (ANR 08-BLAN-0313). Juncheng Wei was supported by a GRF grant
from RGC of Hong Kong. Susanna  Terracini was partially supported by the
Italian PRIN2009 grant ``Critical Point Theory and Perturbative
Methods for Nonlinear Differential Equations". Kelei Wang was
supported by the Australian Research Council.

\end{document}